\documentclass[12pt]{article}
 
\usepackage[margin=1in]{geometry} 
\usepackage{amsmath,amsthm,amsfonts,amssymb,amsthm}
\usepackage{graphicx}
\usepackage{mathrsfs}
\usepackage{mathtools}
\usepackage{hyperref}
\usepackage{float}
\usepackage{authblk}
\usepackage{subfigure}
\usepackage{tikz}
\usepackage{tikz-network}
\usepackage[shortlabels]{enumitem}
\usepackage{wrapfig}
\usepackage{xcolor}

\usetikzlibrary{graphs}
\usetikzlibrary{graphs.standard}

\DeclarePairedDelimiter{\floor}{\lfloor}{\rfloor}

\newtheorem{theorem}{Theorem}[section]
\newtheorem{lemma}[theorem]{Lemma}

\newtheorem{corollary}[theorem]{Corollary}
\newtheorem{conjecture}[theorem]{Conjecture}

\theoremstyle{definition}
\newtheorem{definition}[theorem]{Definition}

\newtheorem{example}[theorem]{Example}

\newenvironment{proofoutline}{\noindent\emph{Proof Outline:}}{\hfill$\square$}

\newcommand{\BZ}{\mathbb Z}

\newcommand{\nice}{\displaystyle}
\newcommand{\COVDS}{\text{COPVD}}
\newcommand{\COEDS}{\text{COPED}}
\newcommand{\COV}[2]{CO^{#1}_{v}(#2)}
\newcommand{\COE}[2]{CO^{#1}_{e}(#2)}

\newcommand{\COVMAX}[1]{CO^{r^+}_{v}(#1)}
\newcommand{\COEMAX}[1]{CO^{r^+}_{e}(#1)}

\newcommand{\COVMIN}[1]{CO^{r^-}_{v}(#1)}
\newcommand{\COEMIN}[1]{CO^{r^-}_{e}(#1)}
\newcommand{\BG}{\mathcal{G}}

\newcommand{\gvMIN}[1]{g_{v}^{r^-}(#1)}
\newcommand{\gvMAX}[1]{g_{v}^{r^+}(#1)}
\newcommand{\geMIN}[1]{g_{e}^{r^-}(#1)}
\newcommand{\geMAX}[1]{g_{e}^{r^+}(#1)}

\begin{document}

\title{On proportional network connectivity}%change title to something more fitting?
\author[1]{Ashley Armbruster}
\author[2]{Jieqi Di}
\author[3]{Nicholas Hanson}
\author[4]{Nathan Shank}
\affil[1]{Frostburg State University}
\affil[2]{Boston College}
\affil[3]{University of Minnesota}
\affil[4]{Moravian College}
\date{\today}

\maketitle

\section{Introduction}

The reliability of a network is an important parameter to consider when building a network. Different characteristics of the network can become unreliable over time or from outside forces. In a simple setting, we model a network as a graph where the vertices represent our objects and a connection between these objects are represented by an edge. These networks could be computer networks, but they could also represent power stations and power lines, people and friendships, or departments in a company and workflow. 
Generally there are two things to consider when discussing the reliability of a network. The first is the conditions that need to be satisfied in order for the network to be operational. We can also consider what properties need to be satisfied in order to render our network inoperable. The second consideration is what properties of our network tend to fail. In certain applications edges are prone to failure and in others vertices are prone to failure. 

One of the first examples of a network reliability measure is the connectivity of a graph, which measures the minimum number of vertices (or edges) which can be removed in order to disconnect the graph. This was also generalized to mixed removals where vertices are removed first followed by edges \cite{beineke_harary}. Network reliability measures have recently been generalized to other graph properties, not just connectivity. For example, in \cite{Gross1998} the authors consider the component order vertex connectivity (COVC) parameter which is the minimum number of vertices which can be removed so that the order of all components in the resulting graph are less than some predetermined value, $k$. A similar measure was developed for edge removals (COEC) in \cite{Gross2006} as well as neighbor removals in \cite{Heinig2015} and \cite{Luttrell2013}.

In this paper we extend this idea. Instead of a network being operational if there is a component of order larger than a fixed size, $k$, we define the network to be operational if there is a component with order greater than some proportion of the original order. So as our networks become larger, we will need a component of proportionally larger order to remain if the network is to be in an operating state. This connectivity measure has been studied for $r=\frac{1}{2}$ and for the purposes of VLSI circuit design \cite{harary}, but we explore the measure for all $0<r<1$.

\section{Definitions}
Throughout the paper we will assume that $G=(V,E)$ is a simple graph with vertex set $V$ and edge set $E$. We will follow the standard graph theory notation found in \cite{Gross2006}. 

Given a graph $G$ define the following: 
\begin{itemize}
  \item Let $V(G)$ be the set of vertices of $G$, and $E(G)$ the set of edges of $G$.  
  \item Let $G-V'$ be the graph obtained from removing the vertices of $V'\subseteq V(G)$ from $G$. Similarly for edges define $G-E'$ for any set $E' \subseteq E(G).$
  
  \item Let $|G|$ be the order of the graph, that is, the cardinality of the vertex set $V(G)$.
\end{itemize}

We will assume throughout that $0 < r < 1$ and $n$ is a positive integer. 

As indicated previously, we are concerned with what determines if our network is operational and what elements of our graph are prone to failure. We will consider a network to be operational if there is a component of sufficiently large size. This is formalized in the following definition. 
\begin{definition} 
Given $0<r<1$, we say a graph $G$ of order $n$ is in an \textit{operating state} if at least one component of $G$ has order greater than $rn$. Otherwise, we will say that $G$ is in a \textit{failure state}. 
\end{definition}

We will consider vertex and edge failures. For vertex removals, if a graph $G$ is in an operating state, we want to find subsets of vertices which can render our graph inoperable. These subsets of vertices will be called \textit{vertex disconnecting sets}.

\begin{definition} 
Given a graph $G=(V,E)$ and $0<r<1$, a subset of vertices $V' \subseteq V$ is a \textit{vertex disconnecting set} if $G-V'$ is in a failure state; that is, every component of $G-V'$ has order at most $rn$, where $n = |V|$. 
\end{definition}

 The set of all vertex disconnecting sets of $G$ will be denoted $\COVDS(G)$. To measure the reliability of the network, we need to find the smallest set of vertices which, when removed, render our network inoperable. This will allow us to quantify the reliability of our network. 

\begin{definition}
Given a graph $G=(V,E)$ and $0<r<1$, the \textit{Component Order Proportion Vertex Connectivity} of $G$, denoted $\COV{r}{G}$, is size of the smallest set $V' \in \COVDS(G)$. If $V' \in \COVDS(G)$ is of minimum size, we will call $V'$ a \textit{minimum vertex disconnecting set.} 
\end{definition}

Now we can consider a similar reliability measure for networks where the edges are prone to failure rather than the vertices. 

\begin{definition} 
Given a graph $G=(V, E)$ and $0<r<1$, a subset of edges $E' \subseteq E$ is an \textit{edge disconnecting set} if $G-E'$ is in a failure state; that is, every component of $G-E'$ has order at most $rn$, where $n = |V|$. 
\end{definition}

The set of all edge disconnecting sets of $G$ will be denoted $\COEDS(G)$. The size of the smallest set in $\COEDS(G)$ will be the measure of the networks reliability. 

\begin{definition}
Given a graph $G=(V,E)$ and $0<r<1$, the \textit{Component Order Proportion Edge Connectivity} of $G$, denoted $\COE{r}{G}$, is size of the smallest set $E' \in \COEDS(G)$. If $E' \in \COEDS(G)$ is of minimum size, we will call $E'$ a \textit{minimum edge disconnecting set}. 
\end{definition}

%%%%%%%%%%%%%%%%%%%%%%%%%%%%%%%%%%%%%%%%%%%%%%%%%%%%%%%%%%%%%%%%%%%%%%%%
%%%%%%%%%%%%%%%%%% General Graph Classes
%%%%%%%%%%%%%%%%%%%%%%%%%%%%%%%%%%%%%%%%%%%%%%%%%%%%%%%%%%%%%%%%%%%%%%%%

\section{COPVC and COPEC in General Graph Classes}

In this section we outline how $\text{CO}_v$ and $\text{CO}_e$ work with several general graph classes including paths, cycles, complete graphs, and complete bipartite graphs.

\subsection{Vertex}
For vertex removal in path graphs, we remove vertices so that the resulting graph will be the disjoint union of the maximum number of large order failure components. This will produce a failure state by removing the fewest vertices.  

%Path COPVC
\begin{theorem} \label{PathCOPVC}
 If $P_n$ is a path graph of order $n$ and $r \in (0,1)$ then  $$CO_v^{r}(P_n)=\left\lfloor \dfrac{n}{\left\lfloor rn\right\rfloor+1} \right\rfloor.$$ 
\end{theorem}
\begin{proof}
Let $P_n = (V, E)$ be a path graph on $n$ vertices defined in the standard way so that $v_1$ and $v_n$ have degree 1 and $v_i$ is adjacent to $v_{i+1}$ for all $1 < i < n$. 

Let $V' = \{v_i \in V: i = j(\floor{rn}+1) \text{ where } j \in \BZ^+ \}$. Note $|V'|=\left\lfloor \frac{n}{\floor{rn}+1} \right\rfloor$.

So $P_n-V' = G_1 + G_2 + \ldots + G_{|V'|+1}$
where $G_i = P_{\floor{rn}}$ except possibly $G_{|V'|+1}$ which will be a subgraph of $P_{\floor{rn}}$. So $P_n-V'$ is the disjoint union of paths of length at most $\floor{rn}$. Since each graph $G_i$ has order at most $\floor{rn}$, we know $V' \in \COVDS(G)$. 

Next we show there is no smaller set in $\COVDS(G)$. 
Assume there is a smaller set $V'' \in \COVDS(G)$ so that $|V''| < |V'|.$
Since $P_n$ is a tree of maximum degree 2, every vertex deletion creates at most one more component. Therefore, $P_{n}-V''$ is the disjoint union of graphs $G'_{1},...,G'_{|V''|+1}$ (allowing empty graphs if necessary). Then we have that 
\begin{align*}
  n=|P_{n}|=|V''| + \sum |G'_{i}|&\le |V''| + \floor{rn}(|V''|+1)\\
  & < |V'|+\floor{rn}|V'|\\
  & \le n.
\end{align*}
So $n<n$, a contradiction. 
\end{proof}

Removing a vertex in a cycle of order $n$ results in a path of order $n-1$. Therefore, we have the following immediate corollary. 

%Cycle COPVC
\begin{corollary}
 For all $r \in (0,1)$, $$CO_v^{r}(C_n)=\left\lfloor \dfrac{n-1}{\left\lfloor r(n-1)\right\rfloor+1} \right\rfloor+1.$$ 
\end{corollary}

Complete graphs are another class of graphs which are easy to understand since any number of vertex removals will not disconnect our graph. 

%Complete COPVC
\begin{theorem}
 For all $r \in (0,1)$,$$\COV{r}{K_{n}}=n-\floor{rn}.$$ 
\end{theorem}

\begin{proof}
Removing any $j$ vertices from a complete graph on $n$ vertices gives a complete graph on $n-j$ vertices. So a failure state for $K_{n}$ will be a complete graph on at most $\floor{rn}$ vertices. So for $V'$ to be in $\COVDS(K_{n})$, it must have at least $n-\floor{rn}$ vertices. So $CO_{v}^{r}(K_{n})=n-\floor{rn}$.
\end{proof}

For complete bipartite graphs, we may produce a failure state by removing one of the vertex sets or a portion of a vertex set.

%Complete Bipartite COPVC
\begin{theorem}
 Assume $a\leq b$. For all $r \in (0,1)$,$$\COV{r}{K_{a,b}}=\min\{a, a+b-\floor{rn}\}.$$ 
\end{theorem}

\begin{proof}
Let $K_{a,b}=(V,E)$ have two parts, $A = \{a_1, a_2, \ldots, a_a\}$ and $B = \{b_1, b_2, \ldots b_b\}$. Let $n=a+b$. Notice that the deletion of any set of vertices will result in a complete bipartite graph or a set of isolated vertices. 

Case 1: Suppose $rn\ge b$. Since $rn\ge b$, $a+b-\floor{rn}\le a$. So let $V'\subseteq A$ be a collection of $a+b-\floor{rn}$ vertices. Notice $|K_{a,b}-V'|=\floor{rn}$ is a failure state, so $V'\in\COVDS(K_{a,b})$. There are no sets of smaller order in $\COVDS(K_{a,b})$, since if $V''\in \COVDS(K_{a,b})$ and $|V''|<|V'|$ then $K_{a,b}-V''$ is connected (since $|V''|<a$) and $|K_{a,b}-V''|>\floor{rn}$. So $V''\notin \COVDS(K_{a,b})$ and $\COV{r}{K_{a,b}}=a+b-\floor{rn}$ in this case.

Case 2: Suppose $rn < b.$ Consider removing $A$. Then $K_{(a,b)} - A$ is an empty graph with $b$ vertices and is in a failure state. For any subset $V'\subseteq V$ with $|V'| <a$ we know $K_{(a,b)}-V'$ is a complete bipartite graph with order $a+b-|V'|$. However, $a+b-|V'| > b$ and therefore, $K_{(a,b)} -V'$ can not be in a failure state. 
\end{proof}

%%%%%%%%%%
\subsection{Edges}
%%%%%%%%%%
We now shift our focus to edge removals for paths, cycles, complete graphs, and complete bipartite graphs. 

%%%% Paths with edge deletions
%Path COEVC
\begin{theorem} \label{PathCOEVC}
 For all $r \in (0,1)$, $\COE{r}{P_n}=\left\lfloor \frac{n-1}{\floor{rn}} \right\rfloor.$
\end{theorem}

\begin{proof}
Let $P_n = (V, E)$ be a defined as in Theorem \ref{PathCOPVC} and label the edge incident to vertices $v_i, v_{i+1}$ as edge $e_i$. . 

Let $E' = \{e_i \in E: i = j\floor{rn} \text{, where } j \in \BZ^+ \}$. Note $|E'|=\left\lfloor \frac{n-1}{\floor{rn}} \right\rfloor$.

So $P_n-E' = G_1 + G_2 + \ldots + G_{|E'|+1}$
where $G_i = P_{\floor{rn}}$ except possibly $G_{|E'|+1}$ which will be a subgraph of $P_{\floor{rn}}$. So $P_n-E'$ is the disjoint union of paths of length at most $\floor{rn}$. Since each graph $G_i$ has order at most $\floor{rn}$, we know $E' \in \COEDS(G)$. 

Next we show there is no smaller set in $\COEDS(G)$. 
Assume there is a smaller set $E'' \in \COEDS(G)$ so that $|E''| < |E'|.$
Since $P_n$ is a tree, every edge deletion creates at most one more component. Therefore, $P_{n}-E''$ is the disjoint union of graphs $G'_{1},...,G'_{|E''|+1}$ (allowing empty graphs if necessary). Then we have that 
\begin{align*}
  n=|E(P_{n})|+1=\sum |E(G'_{i})|+|E''|+1 &\le (\floor{rn}-1)(|E''|+1)+|E''|+1\\
  & < \floor{rn}|E'|+1\\
  & < n.
\end{align*}
So $n<n$, a contradiction. 
\end{proof}

Again, since any vertex remove from a $C_n$ will produce $P_{n-1}$, we have the following: 

\begin{corollary}
 For all $r \in (0,1)$, $$\COE{r}{C_n}=\left\lfloor \frac{n-1}{\floor{rn}} \right\rfloor +1.$$ 
\end{corollary}

\begin{theorem}
For any $r \in (0,1)$, let $n=p\floor{rn}+q$ with $0\leq q < \floor{rn}$. Then
$$\COE{r}{K_{n}}=\binom{n}{2}-p\binom{\floor{rn}}{2} - \binom{q}{2}.$$
\end{theorem}
\begin{proof}

First we prove that for optimal $E'\in \COEDS(K_{n})$, each component of $K_{n}-E'$ is a complete graph. Let $G$ be a component of $K_{n}-E'$ that is not complete. Let $e$ be an edge not present in $E(G)$ whose vertices are in $V(G)$. Since $K_{n}$ is complete, $e\in E(K_{n})$ and so $e\in E'$. Let $E''=E'-\{e\}$. Then note that the components of $K_{n}-E''$ are the same as $K_{n}-E'+e$ since the only difference is the edge $e$, which is contained within the component $G$. Since no components have their order changed, $K_{n}-E''$ is a failure state. So $E''\in \COEDS(K_{n})$ and $|E''|<|E'|$. So a disconnecting set of smaller size can always be found if one of the components of $E'$ is not complete, thus, the removal of a minimum disconnecting set will produce a failure state whose components are all complete graphs.

Let $G_{1},G_{2},\cdots G_{k}$ denote the components of $K_{n}-E'$. Assume $G_{i}$, and $G_{j}$ have order strictly less than $\floor{rn}$ and let $|V(G_i)|\le |V(G_j)|$. Then consider a vertex $v\in V(G_{i})$. Let $E''\subseteq E(K_{n})$ be the same as $E'$ on all vertices $v'\ne v$, but contains the edges connecting $v$ to vertices in $G_i$ and has no edges connecting $v$ to $G_j$. Thus,\\
$E'' = \left(E'-\{e\in E': v \text{ is incident to } e\}\right) \cup \{(v, v'): v' \in V(G_i)\}$.  Now let $G_1',G_2',\cdots$ be the components of $G-E''$. Then $G_k'=G_k$ for $k\ne i,j$,$G_i' = G_i - v$, and $G_j' = G_j + v$. Note that $|G_{i}'|< \floor{rn}$ and $|G_{j}'|\le \floor{rn}$, so $E''\in \COEDS(K_{n})$.
There are $|G_i|-1$ more edges in $E''$ (since we sought to remove those edges from the failure state), but $|G_j|$ less (the edges we sought to add to the failure state). So $|E''| = |E'|-|G_j|+|G_i|-1$ and so $|E''| < |E'|$. Therefore, if there exist two components in $K_{n}-E'$ with order less than $\floor{rn}$, then there exists $E''\in \COEDS(K_{n})$ with strictly lesser cardinality than $E'$.

Therefore, if $n=p\floor{rn}+q$ with $0\leq q < \floor{rn}$ our failure state will consist of $p$ complete graphs of order $\floor{rn}$ and a complete graph of order $q$. 
\end{proof}

The analysis of $\COE{r}{K_{a,b}}$ is more involved. In \cite{COCBipartite} the authors consider the $k$-component order edge connectivity of $K_{a,b}$, denoted $\lambda_c^k(K_{a,b})$, which is defined to be the minimum number of edges that need to be removed so that all of the remaining components have order at most $k-1$. Therefore, if $\floor{rn}=k$, then $\COE{r}{K_{a,b}} = \lambda_c^k(K_{a,b})$.

\section{$G(n,m)$}

We will now turn our attention to the family of graph $G(n,m)$, the set of all graphs with $n$ vertices and $m$ edges.  

\begin{definition}
Let $G(n,m)=\{G: |V(G)|=n, |E(G)|=m \}$ be the set of all graphs of order $n$ and size $m$ for $0 \leq m \leq \binom{n}{2}.$
\end{definition}

In this section we will be concerned with finding the maximum and minimum values of $\COV{r}{G}$ and $\COE{r}{G}$ over all graphs $G \in G(n,m)$. Intuitively, a graph achieves this minimum value would be one of the least reliable networks we can construct with $n$ vertices and $m$ edges.  Similarly, a graph achieves this maximum value would be one of the most reliable networks we can construct with $n$ vertices and $m$ edges.

\begin{definition}
If $\mathcal G$ is a collection of graphs, then $\COVMIN{\mathcal G}= \min\limits_{G \in \mathcal G} \COV{r}{G}$.
\end{definition}

\begin{definition}
If $\mathcal G$ is a collection of graphs, then $\COVMAX{\mathcal G}= \max\limits_{G \in \mathcal G} \COV{r}{G}$.
\end{definition}
  
We can define $\COEMIN{G(n,m)}$ and $\COEMAX{G(n,m)}$ similarly. 

The following lemma gives bounds on the number of edges that we can have in a failure state. In what follows, we can consider a graph of order $n$ and we define $p$ to be the maximum number of failure components of maximum order. Since a component of maximum order has order $\floor{rn}$, $p = \left\lfloor \frac{n}{\floor{rn}}\right\rfloor$. We will let $q$ denote the number of vertices which remain after we remove $p$ pairwise disjoint sets of $\floor{rn}$ vertices. So an induced graph on $p$ vertices will be in a failure state and an induced graph on $q$ vertices will also be in a failure state.  

\begin{lemma}\label{maxedgefailurestate}
Given a graph $G$ with $|V(G)|=n$, define non-negative integers $p$ and $q$ so that $n=\floor{rn}p+q$ with $0 \leq q < \floor{rn}$. If $G$ is in a failure state then $|E(G)|\leq p{\floor{rn} \choose 2}+{q \choose 2}$. In particular, there exists a graph that is the maximal failure state which is the disjoint union of $\floor{rn}$ copies of $K_p$ with one copy of $K_q$. 
%and there exists a graph $G'$ which is in a failure state with $|E(G')| = p{\floor{rn} \choose 2}+{q \choose 2}$. We will call this graph the maximal failure state graph. 
\end{lemma}
\begin{proof}
Assume $n=\floor{rn}p+q$.
Let $G$ be a graph with order $n$ so that $G$ is in a failure state and there does not exist a graph $G'$ of order $n$ which is also in a failure state so that $|E(G')| > |E(G)|$. This means that $G$ has the maximum number of edges possible to be in a failure state. If $G$ contains a component which is not a complete graph, then we can add an edge into that component and $G$ will remain in a failure state. Therefore, $G$ must be the union of components which are complete graphs of order at most $\floor{rn}.$ Assume $G = K_{a_1} \cup K_{a_2} \cup \ldots \cup K_{a_c}$. Then $|E(G)|$ is maximized when $a_1=a_2=\ldots = a_{c-1} = \floor{rn}$ and $a_c=q$, since if this is not the case we have $K_{a}$ and $K_{b}$, $a\leq b < \floor{rn}$. But if this is the case we note

\begin{align*}
  {b \choose 2} + {a \choose 2} &= \frac{b^{2}+b}{2}+\frac{a^{2}+a}{2}\\
  &< \frac{b^{2}+b}{2}+\frac{a^{2}+a}{2} + (b+1) - a\\
  &= \frac{(b+2)(b+1)}{2}+\frac{(a)(a-1)}{2}\\
  &< {b+1 \choose 2}+{a-1 \choose 2}
\end{align*}
and so $K_{a-1}$ and $K_{b+1}$ has more edges than $K_a$ and $K_b$. Therefore, $G$ would not be maximal.
\end{proof}

%%%%%%%START HERE 4/13

\subsection{$\COVMIN{G(n,m)}$}

We will establish some basic properties of $\COVMIN{G(n,m)}$ for a fixed $n$ which will be used to find exact values for $\COVMIN{G(n,m)}.$ 
 
\begin{lemma}[Basic properties of $\COVMIN{G(n,m)}$]\label{covminprop}
Let $n=p\floor{rn}+q$, $0\le q < \floor{rn}$, and let $\gvMIN{m}=\COVMIN{G(n,m)}$. Then the following hold:
\begin{enumerate}[(a)]
  \item $\gvMIN{m}$ is non-decreasing.
  \item $\gvMIN{m+1}-\gvMIN{m} \leq 1$.
  \item $\gvMIN{m}$ takes on the value $0$ for $0 \le m \le p{\floor{rn} \choose 2} + {q \choose 2}$.
  \item The maximum possible value $\gvMIN{m}$ can take on is $n-\left\lfloor\frac{n}{\floor{rn}}\right\rfloor$, and $\gvMIN{m}$ always takes on this value for some $m$.
\end{enumerate}
\end{lemma}

\begin{proof} 
\hfill
\begin{enumerate}[(a)]
  \item Consider the graph $G \in G(n,m)$ where $\COV{r}{G} = \gvMIN{m+1}$. Then removing one edge, $e$, from $G$ clearly does not increase the vertex connectivity since any vertex set whose removal creates a failure state in $G$ will also create a failure state in $G-e$. 
  So $\gvMIN{m} \le \gvMIN{m+1}$.
  \item Consider $G \in G(n,m)$. Then let $D$ be a minimal component order proportion vertex connectivity set of $G$. Now consider some $G' \in G(n,m+1)$ so that $G'=G+e$. Assume $v\in V(G)$ is incident to $e$. Then $D\cup\{v\}$ is a disconnecting set for $G'$ since $G'-D-v$ is a subgraph of $G-D$. Since $\gvMIN{m}$ is non-decreasing, $0 \le \gvMIN{m+1}-\gvMIN{m}\le 1$.
  \item Consider the graph $  [\cup^{p}_{i=1}K_{\floor{rn}} ] \cup K_{q}$, the maximal failure state from Lemma \ref{maxedgefailurestate}. Then this graph is already in a failure state. So by (a), $\gvMIN{m}=0$ for $0\le m \le p {\floor{rn} \choose 2} + {q \choose 2}$.
  % \item Consider the graph $[\cup^{k}_{i=1}K_{p} ] \cup K_{q}$. Then this graph is already in a failure state. So by (a), $g(m)=0$ for $0\le m \le k {p \choose 2} + {q \choose 2}$.
  \item Since $\gvMIN{m}$ is non-decreasing, $\gvMIN{m}$ achieves the greatest value when $m=\binom{n}{2}$, i.e. when $G= K_n$. So the maximum value of $\gvMIN{m}$ is $\COV{r}{K_{n}}=n-\left\lfloor\frac{n}{\floor{rn}}\right\rfloor$.
\end{enumerate}
\end{proof}

%Other attempt
%  $i$ is which full complete graph you are in and $j$ is which vertex you are working on 

%  1 \leq i \leq p-1
%  1 \leq j \leq \floor{rn}
% When j = 1, you have at least one edge on that first vertex. Removing that one edge, reverts back to j = floor{rn} and decrease i. 

Now using the above lemma and Lemma $\ref{maxedgefailurestate}$, we can obtain a full characterization of $\gvMIN{m}=\COVMIN{G(n,m)}$ for fixed $n$ and $r$.

\begin{theorem} \label{covmin}
Fix $n$ and let $0\le k \leq n - \floor{rn}$. Define 

$$f(k) = k(n-k) + \binom{k}{2} + p'\binom{\floor{rn}}{2} + \binom{q'}{2}$$
where $n-k = \floor{rn}p'+q'$ with $0 \leq q' < \floor{rn}$.

Then 
$$
\COVMIN{G(n,m)} = \begin{cases} 
  0 & m \leq f(0)\\
  k & f(k-1) < m \leq f(k).\\
  \end{cases}.
$$
\end{theorem}

\begin{proof}
Fix $0\le k < n - \floor{rn}$. We construct the graph $G$ of order $n$ with maximum size such that $\COV{r}{G}=k$ as follows. We partition the vertex set into two sets, $A$ and $B$, with $|A|=k$ and $|B|=n-k$. Connect every vertex in $A$ to every other vertex in $G$. Also, make the induced subgraph on $B$, denoted $G[B]$, to be the maximal failure state as described in Lemma \ref{maxedgefailurestate}. Clearly the removal of $A$ will put $G$ into a failure state since $G-A = G[B]$. If we removed some vertex set $V'$ with $|V'|=k$ from $G$, so that there is a vertex $b \in B$ in $V'$, then there must be at least one vertex $u\in A$ in $G-V'$. Then note that, since $k<n - \floor{rn}$, $|G-V'|>\floor{rn}$. But since $u$ is connected to every vertex in $G$, it is also connected to every vertex in $G-V'$. So the component $u$ resides in is $G-V'$ itself. But since $|G-V'|>\floor{rn}$, $G-V'$ is not a failure state. Thus, only removing $A$ from $G$ will put it into a failure state, and $\COV{r}{G}=k$.  

For $G$, we can decompose the edge set of $G$ into $E(G)=E(A)\cup E(B)\cup E(A\leftrightarrow B)$ where $E(A \leftrightarrow B)$ is the set of edges incident to one vertex in $A$ and one vertex in $B$. Then consider $G'$ with $|E(G')|>|E(G)|$ and $\COV{r}{G'}=k$. Let $A'$ be the vertex disconnecting set of $G'$ and $B'=V'-A'$. Then let $E(G')=E(A') \cup E(B') \cup E(A'\leftrightarrow B')$. Since $|E(G')|>|E(G)|$, we have that $|E(A')|+|E(B')|+|E(A'\leftrightarrow B')|>|E(A)|+|E(B)|+|E(A \leftrightarrow B)|$. But since $G[A]$ is a complete graph and every vertex of $A$ is adjacent to every vertex of $B$, then $E(A')\ngtr E(A)$ and $E(A' \leftrightarrow B') \ngtr E(A \leftrightarrow B)$. So $E(B')>E(B)$. This contradicts the fact that $G[B]$ is the failure state with the largest possible edge set. So no $G'$ of larger size exists that satisfies $\COV{r}{G'}=k$

The graph $G$ constructed above will have size
$$f(k) = |E(G)| = k(n-k) + \binom{k}{2} + p'\binom{\floor{rn}}{2} + \binom{q'}{2}$$
where $n-k = \floor{rn}p'+q'$ with $0 \leq q' < \floor{rn}$. 

If $m < f(0)$, then it is clear by Lemma \ref{maxedgefailurestate} there exists a graph with of size $m$ which is already in a failure state.

So if $m>f(k)$, then $\COVMIN{G(n,m)}>k$ and since $\COVMIN{G(n,m)}$ is non-decreasing and can only differ by 1 as $m$ increases, we know if $m = f(k), \COVMIN{G(n,m)}=k$. This holds for all $k$, so if $f(k-1) < m \leq f(k)$ the $\COVMIN{G(n,m)} = k$. 

\end{proof}

It can be shown that Theorem \ref{covmin} implies the following results, which allows for a more straightforward computation of $\COVMIN{G(n,m)}$.
\begin{theorem}
Fix $n$ and define $p, q \in \mathbb{Z}$ so that $n=\floor{rn}p+q$ with $0 \leq q < \floor{rn}$. Then consider $0\le m \le {n \choose 2}$. Let $A=p{\floor{rn} \choose 2}+{q \choose 2}$, $B=A+qp\floor{rn}$, and $\nice C_{i,j}=\sum^{i-1}_{t=1}{[\floor{rn}^{2}(p-t)]} + (j-1)\floor{rn}(p-i)+B$. Then label $m$ as follows:

\begin{itemize}
  \item[] If $0 \le m \le A$, then let $f(m)=m_{0}$.
  
  \item[] If $A + tp\floor{rn} < m \le A+(t+1)p\floor{rn}$ for $0 < t \le q$, then let $f(m)=m_{t}$.
  
  \item[] If $C_{i,j}<m\le C_{i,(j+1)} \textrm{ for } 1\le i \le p-1$, $1 \le j < \floor{rn}$, then let $f(m) =m_{i,j}$.
  
  \item[] If $C_{i,\floor{rn}}<m\le C_{i+1,1} \textrm{ for } 1\le i \le p-1$, then let $f(m)=m_{i,\floor{rn}}$.
  
  \item[] If $C_{p,1} \leq m$, then let $f(m) = m_{p,1}$.
  
\end{itemize}
Then we have that

\[ \COVMIN{G(n,m)} =\begin{cases} 
   t & f(m)=m_{t} \\
   (i-1)\floor{rn}+j+q & f(m)=m_{i,j}, i\le p-1 \\
   (p-1)\floor{rn}+q & f(m)=m_{i,j}, i=p.
  \end{cases}
\]

\end{theorem}

\subsection{$\COEMIN{G(n,m)}$}

Now we will find the minimum value of $\COEMIN{G(n,m)}$ which considers edge removals rather than vertex removals. Similar to Lemma \ref{covminprop}, we will first establish some basic properties of $\COEMIN{G(n,m)}$. 

\begin{lemma}\label{coeminprop}(Basic properties of $\COEMIN{G(n,m)}$)  Define non-negative integers $p$ and $q$ so that $n=\floor{rn}p+q$ with $0 \leq q < \floor{rn}$.
For a fixed $n$, let $\geMIN{m}=\COEMIN{G(n,m)}$. Then the following hold:
\begin{enumerate}[(a)]
  \item $\geMIN{m}$ is non-decreasing
  \item $\geMIN{m+1}-\geMIN{m} \leq 1$
  \item $\geMIN{m}$ takes on the value $0$ for $0 \le m \le p{\floor{rn} \choose 2} + {q \choose 2}$
  \item The maximum possible value of $\geMIN{m}$ is ${n \choose 2}-p{\floor{rn} \choose 2} - {q \choose 2}$, and $\geMIN{m}$ always takes on this value for some $m$.
\end{enumerate}
\end{lemma}

\begin{proof} 
\hfill
\begin{enumerate}[(a)]
  \item Consider the graph $G \in G(n,m+1)$ that achieves $\geMIN{m+1}$. Then removing one edge from $G$ clearly does not increase the edge connectivity since the same edge set that creates a failure state in $G$ creates a failure state in $G-e$. So $\geMIN{m} \le \geMIN{m+1}$.
  
  \item Consider $G \in G(n,m)$. Then let $D$ be a minimal component order proportion edge connectivity set of $G$. Now consider some $G' \in G(n,m+1)$ so that $G'=G+e$. Then $D\cup\{e\}$ is a disconnecting set for $G'$ since $G'-D-e=G-D$. Since $\geMIN{m}$ is non-decreasing, $0 \le \geMIN{m+1}-\geMIN{m}\le 1$.
  
  %\item Consider $G \in G(n,m)$. Then let $E$ be the vertex set that creates a failure state in $G$ consisting of components $G_{1},...,G_{b}$. If we add an edge $e$ to $G$, then $E \cup \{e\}$ is clearly a failure state for $G+e$. Applying this to the graph that achieves $\geMIN{m}$ gives us $\geMIN{m+1}-\geMIN{m} \le 1$.
  \item Follows directly from Lemma \ref{maxedgefailurestate}.
  \item Since $g$ is non-decreasing, $K_{n}$ achieves the greatest value of $\geMIN$. Since the maximum failure state is $\displaystyle [\cup^{p}_{i=1}K_{\floor{rn}} ] \cup K_{q}$, we have that the maximum value of $\geMIN{m}$ is ${n \choose 2}-p{\floor{rn} \choose 2} - {q \choose 2}$.
\end{enumerate}
\end{proof}

\noindent We can now find a simple closed form for $\COEMIN{G(n,m)}$.

\begin{theorem}
Fix $n$, let $\geMIN{m}=\COEMIN{G(n,m)}$, and let $n=\floor{rn}p+q$, with $0\le q < \floor{rn}$. Then we have \\
$$\geMIN{m}=
\left\{
\begin{array}{ll}
   0 & m \leq p {\floor{rn} \choose 2}+{q \choose 2} \\ 
   m-p {\floor{rn} \choose 2}-{q \choose 2} & \text{otherwise.} \\ 
\end{array}
\right.$$
\end{theorem}

\begin{proof}
The first case of this piecewise function follows from Lemma \ref{coeminprop}(c). If $m>p{\floor{rn} \choose 2}+{q \choose 2}$, then we can write $m=p{\floor{rn} \choose 2}+{q \choose 2}+i$ for $i>0$. Let $G_{0}$ denote the maximal failure state as described in Lemma \ref{maxedgefailurestate}. Now if $|E(G)|=m$, then if we remove $j<i$ edges we are left with a graph $G'=G-E'$ with $|E(G')|>|E(G_{0})|$, and so $G'$ cannot be a failure state. So $\COVMIN{G(n,m)}\ge i$.

Now we construct a sequence of graphs recursively: let $G_{i}=G_{i-1}+e_{i-1}$ and $G_{1}=G_{0}+e_{0}$, where $e_{i-1}$ is an arbitrary edge not in $G_{i-1}$ and $G_{0}$ is again the maximal failure state. Then $\COE{r}{G_{1}}=1$ since $G_{0}$ is a maximal failure state. Further, if we assume $\COE{r}{G_{i-1}}=i-1$, then $\COE{r}{G_{i}}\le \COE{r}{G_{i}-e_{i-1}}+1=i$. However we know $\COE{r}{G_{i}}\ge i$, and thus, $\COE{r}{G_{i}}=i$ and $\COEMIN{G(n,m)}=i$ in this case.

%Now let $m=k {p \choose 2}+{q \choose 2}+i$, $i>0$. Then consider a collection $E_{i}$ of $i$ distinct edges between the components of $G$. Then $\COE{1/k}{G+E_{i}}=i$, since if $\COE{1/k}{G+E_{i}}<i$ we would have a failure state that has more edges than $G$, which is impossible since $G$ is the maximum failure state. Therefore if $m=k {p \choose 2}={q \choose 2}+i$, $g(m)=i$, completing the proof.
\end{proof}

%\begin{example}
%Similarly to $CO_v^{-\frac{1}{2}(G(n,m)}$, we examine $CO_e^{-\frac{1}{2}(G(n,m)}$ in a graph with a fixed order,, the $x$-axis is $m$ and the $y$-axis is $CO_e^{-\frac{1}{2}(G(n,m)}$. The following is an example function graph with $n=8$.
%\begin{figure}[H]
%\begin{center}
 %  \includegraphics[scale=0.35]{Figures and Images/Graph of COE-.png}\\
%\end{center}
%\end{figure}
%\end{example}

\subsection{$\COVMAX{G(n,m)}$ and $\COEMAX{G(n,m)}$}

Working with $\COVMAX{G(n,m)}$ and $\COEMAX{G(n,m)}$ is considerably more difficult than the minimum functions, and full characterizations of the functions were not found. In this section we present our results as well as routes of attack and conjectures concerning these functions. 

\subsubsection{$\COVMAX{G(n,m)}$}

Similar to the minimum value, $\COVMIN{G(n,m)}$ we have the following properties for our maximum function. These follow a similar argument found in Lemma \ref{covminprop}. 

\begin{lemma}[Basic properties of $\COVMAX{G(n,m)}$]
Let $n=\floor{rn}p+q$, $0\le q < \floor{rn}$, and let $\gvMAX{m}=\COVMAX{G(n,m)}$. Then the following hold:
\begin{enumerate}[(a)]
  \item $\gvMAX{m}$ is non-decreasing.
  \item $\gvMAX{m+1}-\gvMAX{m} \leq 1$.
  \item The maximum possible value $\gvMAX{m}$ can take on is $n-\floor{rn}$, and $\gvMAX{m}$ always takes on this value for some $m$.
\end{enumerate}
\end{lemma}

It is clear that if $m<\floor{rn}$, then we are already in a failure state since the largest connected component could have at most $m+1 \leq \floor{rn}$ vertices. 
Also, if ${n \choose 2} - \floor{rn} < m \le {n \choose 2}$ then $\COVMAX{G(n,m)} = n-\floor{rn}$. So we have bounds on the tails of $\COVMAX{G(n,m)}$ as follows: 

$$\COVMAX{G(n,m)} =
\begin{cases}
   0 &m<\floor{rn}\\
   ? &\floor{rn} \leq m \leq {n \choose 2} - \floor{rn}\\
   n-\floor{rn}, &{n \choose 2} - \floor{rn} < m \le {n \choose 2}.
\end{cases}
$$

% if l < n - \floor{rn}, then CO_{v+}(G(n,m))=l for $\frac{l^2+l\floor{rn}}{2} \le m < \frac{(l+1)^2+(l+1)\floor{rn}}{2}$?
% take l+\floor{rn} vertices, give each one degree l in such a way that it works (i think this is possible!).

However, the case when $\floor{rn} \leq m \leq {n \choose 2} - \floor{rn}$ is not as straight forward because the particular failure state structure is unknown. We do not know if the failure state would consist of a small number of components with a few edges in each component or a single large component with more edges. 

Another way to approach the problem may be to consider what is the minimum number of edge a graph can have if fix $\COV{r}{G} = k$. This could give us a way to find $\COVMAX{G(n,m)}$, however, we were unsuccessful with this approach.

% Start here 4/20

\subsubsection{$\COEMAX{G(n,m)}$}

As in the preceding section, we once again have the following list of simple properties. The proof is essentially identical to Lemma \ref{coeminprop}.

\begin{lemma}[Basic properties of $\COEMAX{G(n,m)}$]
Let $n=\floor{rn}p+q$, $0\le q < \floor{rn}$, and let $\geMAX{m}=\COEMAX{G(n,m)}$. Then the following hold:
\begin{enumerate}[(a)]
  \item $\geMAX{m}$ is non-decreasing.
  \item $\geMAX{m+1}-\geMAX{m} \le 1$.
  \item The maximum possible value $\geMAX{m}$ can take on is ${n \choose 2}-p{\floor{rn} \choose 2} - {q \choose 2}$, and $\geMAX{m}$ always takes on this value for some $m$.
\end{enumerate}
\end{lemma}

For edge remove, it is also clear that we can figure out the tails as follows: 

$$\COEMAX{G(n,m)} = 
\begin{cases}
  0 & \text{ if } m < \floor{rn}\\
  ? & \floor{rn} \leq m < {n \choose 2}\\
  {n \choose 2} - \left(p {\floor{rn} \choose 2} + {q \choose 2}\right) & m={n \choose 2}. \\
  % {n \choose 2} - \left(p {\floor{rn} \choose 2} + {q \choose 2}\right) & {n \choose 2} - \left((p-1)\floor{rn} + q\right) \leq m={n \choose 2} 
\end{cases}
$$

If there are enough edges to guarantee that any subgraph of $G$ will contain $p$ disjoint copies of $K_{\floor{rn}}$ and a disjoint $K_q$, then we know $\COEMAX{G(n,m)} ={n \choose 2} - \left(p {\floor{rn} \choose 2} + {q \choose 2}\right)$ which is the maximum value possible for $\COEMAX{G(n,m)}.$ However, this may not happen for very many values of $m$. For example, assume $m = {n \choose 2} - 2$ and $G = K_{n} - \{e_1, e_2\}$ where $e_1$ and $e_2$ are vertex disjoint edges. Consider the case when $\floor{rn} = n-1$. Then there is no subgraph of $G$ which contains $K_{\floor{rn}}= K_{n-1}$. Therefore, the failure state of $G$ would be $K_{\floor{rn}} - e_1$ and we would have to remove $n-2$ edges, not ${n \choose 2} - \left(p {\floor{rn} \choose 2} + {q \choose 2}\right)= {n \choose 2} - {n-1 \choose 2} = n-1.$

%For a general $r$, the failure state may have $\ceil{1/r}$ components 

\subsubsection{Bounds}

Since a general formula for $\COEMAX{G(n,m)}$ was not found, we would like to bound the function. Given a formula for a general class of graphs immediately gives a lower bound. As an example, Theorem \ref{PathCOEVC} states that $\COE{r}{P_n}=\left\lfloor\frac{n-1}{\floor{rn}}\right\rfloor$ and so $\COEMAX{G(n,m)}\ge \left\lfloor\frac{n-1}{\floor{rn}}\right\rfloor$. However, this method does not give an upper bound.

In the case of $r=\frac{1}{2}$, we attempt to find an upper bound on the $\COEMAX{G(n,m)}$ value by employing a method that looks at bipartite subgraphs in the complement of a graph. We begin with a conjecture about the failure state of a graph that achieves the maximum edge removal connectivity for its graph class.

\begin{conjecture}\label{componentconjecture}
Let $r=\frac{1}{k}$ and let $n$ be a multiple of $k$. The there exists $G\in G(n,m)$ so that $\COE{1/k}{G}=\COE{1/k^+}{G(n,m)}$ and the failure state of $G$ can be written as the union of $k$ components of order $\frac{n}{k}$.
\end{conjecture}

In the case $n=2$, this conjecture is equivalent to the statement that one of the failure states of $G$ is such that the vertices of the components partitions $G$ into two equally sized vertex sets.\\

This conjecture is fruitful when coupled with lower bounds for the largest bipartite subgraph of a graph $G$. Since an edge disconnecting set can be seen as a bipartite graph, minimizing an edge disconnecting set in $G$ is dual to finding the largest bipartite graph in the complement, $\overline{G}$. We can then apply a lower bound to the complement graph class $G(n,\overline{m})$, which gives an upper bound for $G(n,m)$.

\begin{definition}
Let $G$ be a graph. We define the function $b(G)$ to be the size of the largest bipartite subgraph of $G$. For fixed $m$, we define the function $b(m)=\min_{G} b(G)$, where $G \in G(n,m)$.
\end{definition}

The functions $b(m)$ and $b(G)$ have been previously studied.  Some results are summarized below: 

\begin{theorem}[Bounds on the Bipartite Function]\label{bipartitefunction}
Let $G\in G(n,m)$. Then
\begin{itemize}
  \item $b(m) \geq \left\lceil{\frac{1}{2}m +\frac{\sqrt{8m+1}-1}{8}}\right\rceil$. \cite{Edwards1973}
  \item $b(G) \geq \frac{1}{2}(m + \frac{1}{3}n)$ in the case that $G$ has no isolated points. \cite{egk}
  \item $b(G) \geq \frac{1}{2}(m + \frac{1}{2}(n -1))$ in the case that $G$ is a connected graph. \cite{egk}
\end{itemize}
\end{theorem}

These bounds could be used to provide an upper bound for $\COE{1/2^+}{G(n,m)}$ as outlined in the next conjecture.

\begin{conjecture}
Let $n$ be even. Then $\COE{1/2^+}{G(n,m)} \leq \frac{m}{2}+\frac{7n}{12}$.
%\\$\COE{1/2}^+}{G(n,m)} \leq \frac{n^2}{4} - \frac{1}{2}(\bar{m} + \frac{1}{3}n)$.

\end{conjecture}

\begin{proofoutline}
Let $\bar{m}={n \choose 2} - m$. Then from Conjecture \ref{componentconjecture}, let $G\in G(n,m)$ be such that $G$ has maximal $CO^{1/2}_e$ value and has a failure state of $2$ identically sized components. For simplicity, we will assume that $G$ has no isolated vertices. Let $C_1,C_2$ be these components. Then consider the subgraph $H$ consisting of the vertices of $C_1,C_2$ and the edges between the components. Then the number of edges is precisely $\COE{1/2}{G}$. 

%(Note: This is not correct, but the way to think of it is that $|E(B)|$ is the minimum number of edges over all induced subgraphs of $G$ which are bipartipte with 2 equal parts. i.e. you need to include all the edges between $C_1$ and $C_2$ and you want the number of edges to be a minimum. These are the edges you would be removing to put the graph into a failure state.)

Then since $H$ is a bipartite subgraph of $G$, $\COE{1/2}{G}=|E(H)|\le b(G)$. 

Now, we define $\bar{H}$ to be the bipartite complement of $B$; that is, subgraph consisting of vertex sets $C_1$ and $C_2$, with edges connecting the two sets only when an edge is absent in $H$.

%Note that $|E(B)| = \frac{n^2}{4} - |E(\bar{B})|$ where $\bar{H} = K_{\frac{n}{2}, \frac{n}{2}} - B$. Therefore to find the minimum value of $|E(B)|$ we will find the maximum value of $|E(\bar{B})|$. 

Since each bipartite component has order $\frac{n}{2}$, for any bipartite subgraph $B$,\\ $|E(B)|=\frac{n^2}{4}-|E(\bar{B})|$, where $\bar{B}$ is the bipartite complement. Then substituting this for the $|E(B)|$ term in the definition of $b(G)$, we obtain $b(G) = \max_{B\subset G}\left[ \frac{n^2}{4}- \left|E(\bar{B}) \right| \right]$. But this is simply $\frac{n^2}{2}-b(\bar{G})$, and $\COE{1/2}{G}\le b(G) = \frac{n^2}{4}-b(\bar{G})$.

We now apply the second Erd\"os bound from Theorem \ref{bipartitefunction} to $b(\bar{G})$. We obtain the final inequality, $\text{CO}_{e}^{1/2^+}{(G(n,m))}\le \frac{n^2}{4}-\frac{1}{2}(\bar{m} + \frac{1}{3}n)=\frac{m}{2}+\frac{7n}{12}.$ 

If $G$ has $l$ isolated points, we can repeat the previous argument with the isolated points removed and consider $\tilde{n} = n-l$. The upper bound of $\frac{m}{2} + \frac{7n}{12}$ will still hold, as long as $l$ is even.

\end{proofoutline}

%We can use the any of the following bounds for $b(\bar{G})$: (let $\bar{m}=\tbinom{n}{2}-m$)
%\begin{itemize}
%  \item $b_{\bar{\BG}}(m) \geq \frac{1}{2}\bar{m} +\frac{\sqrt{8\bar{m}+1}-1}{8}$ \cite{edwards}
%  \item $b_{\bar{\BG}}(m) \geq \frac{1}{2}(\bar{m} + \frac{1}{3}n')$, in the case that $\bar{G}$ has no isolated points where $n'\leq n$ \cite{egk}
%  \item $b_{\bar{\BG}}(m) \geq \frac{1}{2}(\bar{m} + \frac{1}{2}(n' -1))$, in the case that $\bar{G}$ is connected where $n'\leq n$ \cite{egk}
%\end{itemize}

\section{Acknowledgments}
These results are based on work supported by the National Science Foundation under grants numbered DMS-1852378.

%\begin{thebibliography}{99}

%\end{thebibliography}

\bibliographystyle{amsplain}
\bibliography{mybibfile}

\end{document}